\documentclass[12pt]{amsart}
\usepackage[english]{babel}
\usepackage{amsmath}
\usepackage{amsthm}
\usepackage{amssymb}
\usepackage{mathrsfs}
\usepackage{enumerate}
\usepackage{dsfont}

\newtheorem{theorem}{Theorem}[section]

\newtheorem{proposition}[theorem]{Proposition}
\newtheorem{corollary}[theorem]{Corollary}
\newtheorem{definition}[theorem]{Definition}

\theoremstyle{definition}
\newtheorem{remark}[theorem]{Remark}

\usepackage{xcolor}

\title[Realizations of slice hyperholomorphic functions]{Realizations of holomorphic and
slice hyperholomorphic functions: the Krein space case}
\thanks{Daniel Alpay thanks the Foster G. and Mary McGaw Professorship in
Mathematical Sciences, which supported this research.}
\author[D. Alpay]{Daniel Alpay}
\address{(DA) Schmid College of Science and Technology\\
Chapman University\\
One University Drive
Orange, California 92866\\
USA}
\email{alpay@chapman.edu}
\author[F. Colombo]{Fabrizio Colombo}
\address{(FC) Politecnico di
Milano\\Dipartimento di Matematica\\Via E. Bonardi, 9\\20133 Milano,
Italy}
\email{fabrizio.colombo@polimi.it}
\author[I. Sabadini]{Irene Sabadini}
\address{(IS) Politecnico di
Milano\\Dipartimento di Matematica\\Via E. Bonardi, 9\\20133 Milano\\ Italy}
\email{irene.sabadini@polimi.it}

\begin{document}
\maketitle

\begin{abstract} In this paper we treat realization results for operator-valued functions which are analytic in the complex sense or slice hyperholomorphic over the quaternions.
In the complex setting, we prove a realization theorem for an operator-valued function analytic in a neighborhood of the origin with a coisometric state space
operator thus generalizing an analogous result in the unitary case. A main difference with previous works is the use of reproducing kernel Krein spaces.
We then prove the counterpart of this result in the quaternionic setting. The present work is the first paper which presents a realization theorem with a state space which is a quaternionic Krein space.
\end{abstract}


slice hyperholomorphic functions; quaternionic analysis.

\section{Introduction}
\setcounter{equation}{0}

In this paper we consider realization results for operator-valued functions analytic in the neighborhood of the origin in two cases: the complex number setting and the quaternionic
setting. Recently, the authors proved realization theorems for functions associated to Pontryagin spaces  in the quaternionic setting, see e.g. \cite{acs1,MR3127378} and the
book \cite{zbMATH06658818}. The present work is the first paper which presents a realization theorem with a state space which is a quaternionic Krein space.\\

In the complex setting case, given a Krein space $\mathcal C$ (the coefficient space) we consider a $\mathbf L(\mathcal C,\mathcal C)$-valued function $\Phi$,
analytic in
$\mathbb D_{r_0}=\left\{z\in\mathbb C\,\,\text{{\rm such that}}\,\,|z|<r_0\right\}$
(with $r_0<1$) and continuous in $|z|\le r_0$ in the operator topology, and define
\begin{equation}
\label{huluberlu}
C_\Phi(z,w)=\frac{\Phi(z)+\Phi(w)^{[*]}}{1-z\overline{w}}
\end{equation}
where both $|z|\le r_0$ and $|w|\le r_0$. In \cite{MR903068}
Dijksma, Langer and de Snoo prove that $\Phi$ can be written as
\[
\Phi(z)=i{\rm Im}\, \Phi(0)+C(U+zI)(U-zI)^{-1}C^{[*]}
\]
where $U$ is a {\sl bounded} unitary operator in a Krein space $\mathcal K$ and $C$ is a bounded linear map from $\mathcal K$ into the coefficient space $\mathcal C$
(note that an everywhere defined unitary map in Krein space need not be continuous), and $[*]$ denotes Krein space adjoints; see also \cite{cdls,MR89a:47055} for related works.
A key tool in their argument is a result of Krein on boundedness of operators in
Hilbert spaces endowed with additional Hermitian forms; see
\cite{MR0187086}, \cite[p. 75]{MR92m:47068},
\cite{MR0024574}, \cite{MR0068116},  \cite{MR0045314}.\smallskip

We prove a similar result with the unitarity constraint
replaced by a coisometry constraint. Our method is different from \cite{MR903068} although, as in that paper (\cite[p. 133]{MR903068})
we use the sesquilinear form \eqref{phedre}. This form defines a bounded self-adjoint operator, and we associate to this operator
a reproducing kernel Krein space with reproducing kernel
\eqref{huluberlu}. This could be done on general grounds since, in view of the Krein space structure, both
\[
\frac{(I+\Phi(z))(I+\Phi(w)^{[*]})}{2(1-z\overline{w})}\quad{\rm and}\quad \frac{(I-\Phi(z))(I-\Phi(w)^{[*]})}{2(1-z\overline{w})}
\]
are differences of two positive definite functions, and so is
\[
C_\Phi(z,w)=\frac{(I+\Phi(z))(I+\Phi(w)^{[*]})-(I-\Phi(z))(I-\Phi(w)^{[*]})}{2(1-z\overline{w})}.
\]
Using a result of Laurent Schwartz, see \cite{schwartz}, one can assert that there exists an associated reproducing kernel
Krein space of vector-valued functions analytic in $\mathbb D_r$. One can also use arguments as in \cite{a2,Alpay92a}. Here we
construct a backward-shift invariant space. The boundedness questions are now not considered
using the above mentioned result of Krein, but using the reproducing kernel property (see
Proposition \ref{closed}) and Loewner's theorem. We take this opportunity to recall that
under the hypothesis of the countable axiom of choice, all (linear) everywhere defined operators in Hilbert
space are bounded; see \cite{MR48:6991}. The non continuous everywhere defined unitary maps mentioned above are not closed.\\

The paper consists of seven sections, this introduction being the first. The next three sections focus on the complex setting case.
We review some results on operator ranges in Section \ref{lydia}. Preliminary results,
and in particular the study of a certain associated Hermitian form, are gathered in Section \ref{sec3}. The realization theorem itself is proved in Section \ref{sec4}.
The last three sections are devoted to the quaternionic setting. Some facts on slice hyperholomorphic functions are recalled in Section \ref{quater}. In Section \ref{quater2}
we study an Hermitian form used to prove the quaternionic version of the realization theorem. This theorem is proved in turn in Section \ref{sec7}.

\section{Operator ranges}
\setcounter{equation}{0}
\label{lydia}
Operator ranges are a main tool in our construction, and in this section we discuss some relevant results  in the complex setting.
We refer in particular \cite{a1,a2,Alpay92a,zbMATH06526205}. The quaternionic case is postponed to Section \ref{quater}.\\

Let $\mathcal H$ be a Hilbert space over the complex numbers, and let $P$ denote a bounded Hermitian operator in $\mathcal H$.
We set
\[
P=\sigma|P|
\]

to be the polar decomposition of $P$, where $|P|$ is the absolute value of $P$
and $\sigma$ its sign. We denote by $|P|^{1/2}$ the unique positive squareroot of $|P|$, and endow ${\rm ran}~|P|^{1/2}$ with the following two forms:
\begin{align}
\label{shulamit111}
\langle |P|^{1/2}f,|P|^{1/2}g\rangle_P&=\langle f,(I-\pi)g\rangle_{\mathcal H}\\
[|P|^{1/2}f,|P|^{1/2}g]_P&=\langle \sigma f,(I-\pi)g\rangle_{\mathcal H},
\label{eglantine}
\end{align}
where $\pi$ is the orthogonal projection on $\ker P$. Setting
\[
P=|P|^{1/2}|P|^{1/2}\sigma,
\]
we note now that $|P|^{1/2}\pi=0$. Indeed, $\ker |P|=\ker P$ since $|P|=\sigma P$ and $\ker |P|=\ker |P|^{1/2}$ by e.g. the Cauchy-Schwarz inequality or the spectral theorem.
Hence, it holds that
\begin{align}
\label{michelle}
[|P|^{1/2}f,Pg]_P&=\langle |P|^{1/2}f, g\rangle_{\mathcal H}\\
\langle |P|^{1/2}f,Pg\rangle_P&=\langle f,|P|^{1/2}\sigma g\rangle_{\mathcal H}.
\label{michelle1}
\end{align}
Furthermore, on ${\rm ran}\,P$ the two Hermitian forms above take the form
\begin{align}
[Pf,Pg]_P&=\langle Pf,g\rangle_{\mathcal H}\\
\langle Pf,Pg\rangle_P&=\langle |P|f,g\rangle_{\mathcal H}.
\label{rtyui}
\end{align}

We refer to \cite{a1,zbMATH06526205} for a proof of the following proposition.

\begin{proposition}
The space $({\rm ran}\,P,\langle \cdot,\cdot\rangle_P)$ is a pre-Hilbert space whose closure is ${\rm ran}\,|P|^{1/2}$, with inner product
\eqref{shulamit111}. Furthermore the space ${\rm ran}\,|P|^{1/2}$ is a Krein space when
endowed with the form \eqref{eglantine}, and we have
\[
\langle |P|^{1/2}f,\sigma |P|^{1/2}g\rangle_P=[|P|^{1/2}f,|P|^{1/2}g]_P,\quad f,g\in\mathcal H.
\]
\label{milano1234}
\end{proposition}

In preparation for the next result, we recall that the Aronszajn-Moore one-to-one
correspondence between positive definite functions and
reproducing kernel Hilbert spaces (see \cite{aron})
does not extend in a straightforward way to the case of Hermitian functions. As was proved by L. Schwartz, see \cite{schwartz}, there is an onto, but not one-to-one, map from the
set of reproducing kernel Krein spaces of functions on a given set and the set of differences of positive definite functions on this given set.

\begin{definition}
  Let $(\mathcal C,[\cdot,\cdot]_{\mathcal C})$ be a Hilbert space, and let $\mathcal H$ be a Hilbert space of $\mathcal C$-valued functions.
Let $(g_z)_{z\in \Omega}$ be a family of operators from $\mathcal C$ into $\mathcal H$ whose ranges span a dense subspace of $\mathcal H$.
For $f\in \mathcal H$ and $\eta\in\mathcal C$ we define an ``associated transform'' of $f$ denoted by $\widehat{|P|^{1/2}f}$via
\begin{equation}
\label{widehat}
[\widehat{|P|^{1/2}f}(z),\eta]_{\mathcal C}=[ |P|^{1/2}f,P( g_z\eta)]_P=\langle |P|^{1/2}f,g_z\eta\rangle_{\mathcal H}.
\end{equation}
\end{definition}

\begin{proposition}
\label{sar123}
In the setting of the previous definition, the set of functions $\widehat{|P|^{1/2}f}$ defined by \eqref{widehat} with the inner product
\begin{equation}
\label{rtyuiop123}
[\widehat{|P|^{1/2}f},\widehat{|P|^{1/2}g}]_{K}=[|P|^{1/2}f,|P|^{1/2}g]_P
\end{equation}
is a reproducing kernel Krein space with reproducing kernel equal to
\[
K(z,w)\xi=(\widehat{P(g_w\xi)})(z)
\]
\label{z1z2z3}
\end{proposition}

\begin{proof}
Using \eqref{michelle} to go from the second line to the third in the following
computations one we can write

\[
\begin{split}
[\widehat{|P|^{1/2}f},K(\cdot,w)\eta]_{K}&=[|P|^{1/2}f,P(g_w\eta)]_P\\
&=[|P|^{1/2}f,|P|^{1/2}(|P|^{1/2}\sigma g_w\eta)]_P\\
&=\langle |P|^{1/2}f, g_w\eta)\rangle_{\mathcal H}\\
&=[ \widehat{|P|^{1/2}f}(w),\eta]_{\mathcal C}.
\end{split}
\]
\end{proof}

We note that
\begin{equation}
\label{esther}
[K(z,w)\xi,\eta]_{\mathcal C}=[P(g_w\xi),P(g_z\eta)]_P
\end{equation}
and, by replacing $f$ by $|P|^{1/2}\sigma f$ in \eqref{rtyuiop123},
\begin{equation}
\label{esther1}
[\widehat{Pf}(z),\eta]_{\mathcal C}=[Pf,P(g_z\eta)]_P=\langle Pf,g_z\eta\rangle_{\mathcal H}.
\end{equation}
These last equalities are used in particular in the proof of Proposition \ref{prop42}.

\section{The complex variable setting: preliminaries}
\setcounter{equation}{0}
\label{sec3}
In this section we introduce a space, denoted by $\mathbf H_{2,r}^-(\mathcal C)$, which turns out to be a reproducing kernel Hilbert space. We equip this space with a suitable sesquilinear form and we prove some useful properties. As in the previous section we consider the case where the coefficient space is a Hilbert space.
The case of a Krein space is treated in Remark \ref{lastrem}.\\
In the sequel, let $(\mathcal C,\langle\cdot,\cdot\rangle_{\mathcal C})$ be a Hilbert space space, and $\|\cdot\|_{\mathcal C}$ denote the associated norm.

\begin{definition} With the above notation,
and with $R=1/r$, $0<r<1$, we denote by
$\mathbf H_{2,r}^-(\mathcal C)$ the space of power series of the form
\begin{equation}
\label{marseille}
  f(z)=\sum_{u=1}^\infty \frac{f_u}{z^u},
\end{equation}
where the coefficients $f_1,f_2\ldots \in\mathcal C$ and satisfy
\begin{equation}
\label{concorde}
\sum_{u=1}^\infty R^{2u}\|f_u\|_{\mathcal C}^2<\infty,
\end{equation}
with associated inner product
\begin{equation}
\label{innerprod}
\langle f,g\rangle_{\mathbf H_{2,r}^-(\mathcal C)}=\sum_{u=1}^\infty R^{2u}\langle f_u,g_u\rangle_{\mathcal C}\quad ( where\,\, g(z)=\sum_{u=1}^\infty\frac{g_u}{z^u}).
\end{equation}
We denote by $\ell_{2,r}(\mathbb N,\mathcal C)$ the space of vectors
\[
{\mathbf f}=\begin{pmatrix}f_1\\ f_2\\ \vdots\end{pmatrix}\in{\mathcal C}^{\mathbb N}
\]
such that \eqref{concorde} holds.
\label{shulamit}
\end{definition}
\begin{proposition}
Elements of $\mathbf H_{2,r}^-(\mathcal C)$ are analytic in $|z|>r$, and $\mathbf H_{2,r}^-(\mathcal C)$ is a reproducing kernel Hilbert space of $\mathcal C$-valued functions
with reproducing kernel
\begin{equation}
k(z,w)=\frac{r^2I_{\mathcal C}}{z\overline{w}-r^2}.
\label{rkz}
\end{equation}
For $\xi\in\mathcal C$ and $\nu\in\mathbb C$ such that $|\nu|<r$ the function
\begin{equation}
u\mapsto \frac{\xi}{u-\nu}
\end{equation}
belongs to $\mathbf H_{2,r}^-(\mathcal C)$.
\end{proposition}

\begin{proof}
Taking into account the power expansion
\begin{equation}
k(z,w)=\frac{r^2I_{\mathcal C}}{z\overline{w}-r^2}=\frac{r^2I_{\mathcal C}}{z\overline{w}}\frac{1}{1-\frac{r^2}{z\overline{w}}}=\sum_{u=1}^\infty\frac{r^{2u}}{\overline{w}^u}\frac{1}{z^u}
\label{rkz1}
\end{equation}
and using the definition \eqref{innerprod} of the inner product we have for $f$ of the form \eqref{marseille} and $\eta\in\mathcal C$:
\[
\begin{split}
  \langle f(\cdot),k(\cdot,w)\eta \rangle_{\mathbf H_{2,r}^-(\mathcal C)}&=\sum_{u=1}^\infty R^{2u}\langle f_u,\eta\rangle_{\mathcal C}
  \frac{r^{2u}}{\overline{w}^u}\\
&=\langle f(w),\eta\rangle_{\mathcal C}.
\end{split}
\]

For $|z|>r$ we have
\[
\frac{1}{z-\nu}=\frac{1}{z(1-\frac{\nu}{z})}=\sum_{u=0}^\infty\frac{\nu^{u}}{z^{u+1}},
\]
where $f_u=\nu^{u-1}$ for $u=1,2,\ldots$ and
\[
\sum_{u=1}^\infty R^{2u}|\nu|^{2(u-1)}=R^2\cdot\sum_{u=0}|\frac{\nu}{r}|^{2u}<\infty .
\]
\end{proof}

Set now $1>r_0>r$.
We consider $\Phi$ a $\mathbf L(\mathcal C,\mathcal C)$-valued function analytic in $\mathbb D_{r_0}$
and continuous in $|z|\le r_0$ in the operator topology, and we define
$C_\Phi$ as in \eqref{huluberlu}:
\[
C_\Phi(u,v)=\frac{\Phi(u)+\Phi(v)^*}{1-u\overline{v}}
\]
where both $|u|\le r_0$ and $|v|\le r_0$. If we set
\[
M=\max_{|z|\le r_0}\|\Phi(z)\|,
\]
then
\[
\|C_{\Phi}(u,v)\|\le \frac{2M}{1-r_0^2}.
\]

For a fixed function $\Phi$ as above we define a sesquilinear form on
$\mathbf H_{2,r}^-(\mathcal C)$ by:
\begin{equation}
[f,g]_{\Phi}=\frac{1}{4\pi^2}\iint_{\substack{|a|=r\\|b|=r}}\left[C_\Phi(a,b)f(a),g(b)\right]_{\mathcal C}dad\overline{b}.
\label{phedre}
\end{equation}
For similar forms, see \cite[p. 133]{MR903068} and \cite[p. 1199]{a2}. We now study the properties of this form needed to prove the realization result in the next section.

\begin{proposition}
Let $f,g\in\mathbf H_{2,r}^-(\mathcal C)$ equipped with the sesquilinear form
\eqref{phedre}. It holds that
\begin{equation}
|[f,g]_{\Phi}|\le
 \frac{2Mr_0^2}{(1-r_0^2)^2}\left(\sum_{u=1}^\infty R^{2n}\|f_u\|^2\right)^{1/2}\left(\sum_{u=1}^\infty R^{2u}\|g_u\|^2\right)^{1/2}
\label{bois_de_vincennes}
\end{equation}
and in particular $[\cdot,\cdot]_{\Phi}$ is jointly continuous with respect to the topology of $\mathbf H_{2,r}^-(\mathcal C)$.
\end{proposition}

\begin{proof}
\[
\begin{split}
|[f,g]_{\Phi}|&\le\frac{2M}{1-r_0^2}\left(\int_{|a|=r}\|f(a)\|da\right)\left(\int_{|b|=r}\|g(b)\|db\right)\\
&\le\frac{2M}{1-r_0^2}\left(\int_{|a|=r}\sum_{u=1}^\infty r^u\|f_u\|\right)\left(\int_{|b|=r}\sum_{u=1}^\infty r^u\|g_u\|\right).
\end{split}
\]
But
\[
\begin{split}
\sum_{u=1}^\infty r^u\|f_u\|&=\sum_{u=1}^\infty r^{2u}R^u\|f_u\|\\
&\le \left(\sum_{u=1}^\infty r^{4n}\right)^{1/2}\left(\sum_{u=1}^\infty R^{2u}\|f_u\|^2\right)^{1/2}\\
&\le \left(\sum_{u=1}^\infty r_0^{4u}\right)^{1/2}\left(\sum_{u=1}^\infty R^{2u}\|f_u\|^2\right)^{1/2}
\end{split}
\]
and similarly for $g$. This concludes the proof.
\end{proof}

In view of the preceding proposition,
Riesz representation theorem allows to define in a unique way an Hermitian and everywhere defined operator $P$ such that
\begin{equation}
\label{quai-de-la-rapee}
[f,g]_{\Phi}=\langle Pf,g\rangle_{\mathbf H_{2,r}^-(\mathcal C)}=[Pf,Pg]_P,
\end{equation}
The continuity of $P$ follows from a well known result from functional analysis; see e.g \cite[Exercise 4.2.19 p. 212]{CAPB_2}.

Note that
\begin{equation}
(Pf)({b})=\frac{-r^2}{2\pi i b}\int_{|a|=r}C_\Phi(a,b)f(a)da
\label{richelieu_drouot}
\end{equation}
\begin{proposition}
Let $P$ given by \eqref{richelieu_drouot}.
Then,
\begin{equation}
\langle Pf,g\rangle_{\mathbf H_{2,r}^-(\mathcal C)}=\frac{1}{4\pi^2}\iint_{\substack{|a|=r\\|b|=r}}\langle C_\Phi(a,b)f(a),g(b)\rangle_{\mathcal C}dad\overline{b}.
\end{equation}
Furthermore, it holds that
\begin{equation}
\label{johanna}
\left(P\left(\frac{\xi}{a-\overline{w}}\right)\right)(b)=\frac{-r^2}{b}\frac{\Phi({b})^*+\Phi(\overline{w})}{1-\overline{b}\overline{w}}\xi,\quad\xi\in\mathcal C, \quad |b|=r.
\end{equation}
\end{proposition}

\begin{proof}
We have
\[
\begin{split}
\langle Pf,g\rangle_{\mathbf H_{2,r}^-(\mathcal C)}&=\frac{1}{2\pi i}\int_{|b|=r}[Pf(b),g(b)]_{\mathcal C}\frac{db}{b}\\
&=\frac{r^2}{4\pi^2  }\int_{|b|=r}\frac{1}{b}\left[\int_{|a|=r}C_\Phi(a,b)f(a)da,g(b)\right]_{\mathcal C}\frac{db}{b}\\
&=\frac{r^2}{4\pi^2  }\iint_{\substack{|a|=r\\|b|=r}}\left[C_\Phi(a,b)f(a),g(b)\right]_{\mathcal C}\frac{dadb}{b^2},
\end{split}
\]
where continuity justifies the interchanging of inner product and integral to go from the second to the third line above.
The result follows from
\[
d\overline{b}=r^2\frac{db}{b^2}.
\]
To prove the second claim, we write
\[
\begin{split}
\left(P\left(\frac{\xi}{a-\overline{w}}\right)\right)(b)&=-\frac{r^2}{2\pi i b}\int_{|a|=r}\frac{\Phi(a)\xi+\Phi(b)^*\xi}{(1-a\overline{b})(a-\overline{w})} da\\
\end{split}
\]
and the result follows from Cauchy's formula.
\end{proof}

The range of $P$ is inside $\mathbf H_{2,r}^-(\mathcal C)$, and so the right side
of \eqref{johanna} is the restriction to $|b|=r$ of a function in the variable $b$
in  $\mathbf H_{2,r}^-(\mathcal C)$. To verify this directly we note that
\[
\frac{\Phi({b})^{*}+\Phi(\overline{w})}{1-\overline{b}\overline{w}},\quad |b|=r,
\]
is a power series in $\overline{b}$, with possibly a constant term.  Note that $\overline{b}=r^2/b$ has modulus strictly less than $r$ for $|b|>r$, and so
$\Phi(b)^*$ extend to a function analytic in $|b|>r$. Together with the factor $-r^2/b$ in front of the right side of \eqref{johanna} this leads to the conclusion.\\

\begin{proposition}
Let $f,g\in\mathbf H_{2,r}^-(\mathcal C)$. It holds that
\begin{equation}
\label{elizabeth}
\begin{split}
\iint_{\substack{|a|=r\\|b|=r}}
\left[C_\Phi(a,b)(af(a)-f_1),g(b)\right]_{\mathcal C}dad\overline{b}&=\\
&\hspace{-5cm}=
\iint_{\substack{|a|=r\\|b|=r}}\left[C_\Phi(a,b)f(a),b^{-1}g(b)\right]_{\mathcal C}da
d\overline{b},
\end{split}
\end{equation}
or, with some abuse of notation,
\begin{equation}
[ (af(a)-f_1),g(b)]_{\Phi}=[f(a),b^{-1}g(b)]_{\Phi}.
\end{equation}
\end{proposition}

\begin{proof} We have:
\[
\begin{split}
4\pi^2[ (af(a)-f_1),g(b)]_{\Phi}-[f(a),b^{-1}g(b)]_{\Phi}&=\\
&\hspace{-4cm}=\iint_{\substack{|a|=r\\|b|=r}}
\left(a-\frac{1}{\overline{b}}\right)[C_\Phi(a,b)f(a),g(b)]_{\mathcal C}dad\overline{b}\\
&\hspace{-3.5cm}-\iint_{\substack{|a|=r\\|b|=r}}[C_\Phi(a,b)f_1,b^{-1}g(b)]_{\mathcal C}dad\overline{b}\\
&\hspace{-4cm}=\iint_{\substack{|a|=r\\|b|=r}}[
\left(\Phi(a)+\Phi(b)^*\right)f(a),b^{-1}g(b)]_{\mathcal C}dad\overline{b}\\
&\hspace{-3.5cm}-\iint_{\substack{|a|=r\\|b|=r}}[C_\Phi(a,b)f_1,b^{-1}g(b)]_{\mathcal C}dad\overline{b}\\
&\hspace{-4cm}=[\int_{|a|=r}\Phi(a)f(a)da,\int_{|b|=r}b^{-1}g(b)d\overline{b}]_{\mathcal C}+\\
&\hspace{-3.5cm}+[\int_{|a|=r}f(a)da,\int_{|b|=r}\Phi(b)b^{-1}g(b)d\overline{b}]_{\mathcal C}-\\
&\hspace{-3.5cm}-[\int_{|a|=r}\frac{\Phi(a)f_1}{1-a\overline{b}}da,\int_{|b|=r}b^{-1}g(b)d\overline{b}]_{\mathcal C}-\\
&\hspace{-3.5cm}-[\int_{|a|=r}\frac{f_1}{1-a\overline{b}}da,\int_{|b|=r}\Phi(b)b^{-1}g(b)d\overline{b}]_{\mathcal C}\\
&\hspace{-4cm}=0
\end{split}
\]
since
\[
\int_{|a|=r}f(a)da=\int_{|a|=r}\frac{f_1}{1-a\overline{b}}da\quad{\rm and}\quad \int_{|b|=r}b^{-1}g(b)d\overline{b}=0.
\]
\end{proof}

\begin{corollary}
Let  $T$ be the operator which to $f\in\mathbf H_{2,r}^-(\mathcal C)$ associates the function
\begin{equation}
\label{TTT}
a\,\mapsto\,af(a)-f_1,
\end{equation}
and let $\mathsf{M}_{b^{-1}}$ be the operator of multiplication by $b^{-1}$.
Then $T$ and $\mathsf{M}_{b^{-1}}$ are bounded operators and we have
\begin{equation}
\label{muriel}
\langle PTf,g\rangle_{{\mathbf H}_{2,r}^-(\mathcal C)}=\langle Pf,\mathsf{M}_{b^{-1}}g\rangle_{{\mathbf H}_{2,r}^-(\mathcal C)}.
\end{equation}
\end{corollary}

\begin{proof}
The assertion follows from \eqref{quai-de-la-rapee} and  the preceding  proposition.
\end{proof}


In particular, and with adjoints in the Hilbert space ${\mathbf H}_{2,r}^-(\mathcal C) $
\begin{equation}
\label{solange}
T^*P=P\mathsf{M}_{b^{-1}}.
\end{equation}

We now express the operator $P$ in terms of the coefficients of the power series expansion of $\Phi$. This is not needed to prove Theorem \ref{th55}, but will be crucial in
the quaternionic setting.

\begin{proposition}
\label{illusion}
Let $\Phi(z)=\sum_{u=0}^\infty \Phi_uz^u$, with $\Phi_u\in\mathbf L(\mathcal C,\mathcal C)$. Then,
\begin{equation}
\label{sofsoi}
\begin{split}
\frac{1}{4\pi^2}\iint_{\substack{|a|=r\\|b|=r}}\left[C_\Phi(a,b)f(a),g(b)\right]_{\mathcal C}dad\overline{b}&=\sum_{v=1}^\infty\sum_{r=1}^v \langle \Phi_{v-r}f_v,g_r\rangle_{\mathcal C}+\\
&\hspace{5mm}+
\sum_{u=1}^\infty\sum_{r=1}^u \langle \Phi^*_{u-r}f_r,g_u\rangle_{\mathcal C}.
\end{split}
\end{equation}
where $f(z)=\sum_{v=1}^\infty \frac{f_v}{z^v}$ and $g(z)=\sum_{u=1}^\infty \frac{g_u}{z^u}$ belong to $\mathbf H_{2,r}^-(\mathcal C)$.
\end{proposition}

\begin{proof}
We have
\[
C_\Phi(a,b)=\sum_{t=0}^\infty a^t\overline{b}^t(\sum_{s=0}^\infty \Phi_s a^s+\Phi^*_s\overline{b}^s)
\]
and so
\[
\begin{split}
\langle C_\Phi(a,b)f(a), g(b)\rangle_{\mathcal C}&=\sum_{u=1;v=1;t=0;s=0}^\infty \overline{b}^{-u}\overline{b}^ta^{-v}a^ta^s\langle \Phi_sf_v , g_u\rangle_{\mathcal C}+\\
&\hspace{5mm}+\sum_{u=1;v=1;t=0;s=0}^\infty \overline{b}^{-u}\overline{b}^{s}\overline{b}^ta^ta^{-v}\langle \Phi_s^*f_v, g_u\rangle_{\mathcal C}.
\end{split}
\]
When computing the integral
\[
\frac{1}{4\pi^2}\iint_{\substack{|a|=r\\|b|=r}}\left\{\cdot\right\}dad\overline{b},
\]
in the first sum, only the terms for which
\[
t-u=-1\quad{\rm and}\quad t+s-v=-1
\]
lead to a possibly nonzero contribution, while in the second sum the corresponding indices are
\[
t+s-u=-1\quad{\rm and}\quad t-v=-1.
\]
So \eqref{sofsoi} holds, that is:
\[
\begin{split}
\frac{1}{4\pi^2}\iint_{\substack{|a|=r\\|b|=r}}\left[C_\Phi(a,b)f(a),g(b)\right]_{\mathcal C}dad\overline{b}&=\sum_{v=1}^\infty\sum_{r=1}^v \langle \Phi_{v-r}f_v,g_r\rangle_{\mathcal C}+\\
&\hspace{5mm}+
\sum_{u=1}^\infty\sum_{r=1}^u \langle \Phi^*_{u-r}f_r,g_u\rangle_{\mathcal C}.
\end{split}
\]
\end{proof}

We note that \eqref{bois_de_vincennes} expresses the fact that the lower triangular block-matrix
\begin{equation}
T_\Phi=\begin{pmatrix}\Phi_0&0&0&\cdots\\
2\Phi_1^*&\Phi_0&0&\cdots\\
2\Phi_2^*&2\Phi_1^*&\Phi_0&0&\cdots\\
\vdots&      & \ddots     &\ddots&0&\cdots
\end{pmatrix}
\label{ttt}
\end{equation}
defines a bounded (block-Toeplitz) operator from $\ell_{2,r}(\mathbb N,\mathcal C)$ (see Definition \ref{shulamit} for the latter) into itself.
Then the right side of \eqref{sofsoi} can be rewritten as
\begin{equation}
\langle ({\rm Re}\,T_\Phi)\,{\mathbf f},{\mathbf g}\rangle_{\ell_{2,r}(\mathbb N,\mathcal C)},
\end{equation}
with
\[
{\mathbf f}=\begin{pmatrix}f_1\\ f_2\\ \vdots\end{pmatrix}\quad{\rm and}\quad {\mathbf g}=\begin{pmatrix}g_1\\ g_2\\ \vdots\end{pmatrix}.
\]
By uniqueness of the operator $P$ defined in \eqref{quai-de-la-rapee}, we see that ${\rm Re}\,T_\Phi$ is the matrix representation of $P$ in the standard orthonormal basis of
$\mathbf H_{2,r}^-(\mathcal C)$.
\section{The complex variable setting: The realization theorem}
\setcounter{equation}{0}
\label{sec4}
We now apply the results of Section \ref{lydia} to
$\mathcal H=\mathbf H_{2,r}^-(\mathcal C)$ and to the operator $P$ defined by \eqref{richelieu_drouot}. According to Proposition \ref{milano1234}, the space ${\rm ran}\, |P|^{1/2}$ endowed with the
inner product \eqref{eglantine} is a Krein space of functions analytic in $|z|>r$. To get a Krein space of functions analytic in $\mathbb D_r$ with the required reproducing kernel we
define an associated transform \eqref{widehat} and take $\Omega=\mathbb D_r$. We choose
the family $(g_z)_{z\in \mathbb D_r}$ to be
\[
g_z\,\,:\,\,b\mapsto \frac{I_{\mathcal C}}{b-\overline{z}},\quad z\in\mathbb D_r.
\]

Formula \eqref{esther} gives (see also \cite[p. 133]{MR903068})
\begin{equation}
\label{euxodice}
\begin{split}
[\frac{\xi}{a-\overline{w}},\frac{\eta}{b-\overline{z}}]_{\Phi}&=\langle \left(P\frac{\xi}{\cdot-\overline{w}}\right)(b),\frac{\eta}{b-\overline{z}}
\rangle_{\mathbf H_{2,r}^-(\mathcal C)}\\
&=
[\frac{\Phi(\overline{z})^*+\Phi(\overline{w})}{1-z\overline{w}}\xi,\eta]_{\mathcal C},\quad |z|<r,\,\,\, |w|<r,
\end{split}
\end{equation}

\begin{definition}
The space consisting of the functions $F$ such that
\begin{equation}
[F(z),\eta]_{\mathcal C}=\langle(|P|^{1/2}f)(b),\frac{\eta}{b-\overline{z}}\rangle_{\mathbf H_{2,r}^-(\mathcal C)},\quad f\in \mathbf H_{2,r}^-(\mathcal C),
\end{equation}
built from Proposition \ref{sar123} is {\sl a} reproducing kernel Krein space with reproducing kernel defined by \eqref{euxodice}, which
we will denote by $\mathcal L(\Phi^\sharp)$ with $\Phi^\sharp(z)=(\Phi(\overline{z}))^*$.
\end{definition}
Note that
\begin{equation}
\label{shoterazoulai}
[F(z),\eta]_{\mathcal C}=[f,\frac{\eta}{b-\overline{z}}]_{\Phi},\quad f\in \mathbf H_{2,r}^-(\mathcal C),
\end{equation}
and that \eqref{shoterazoulai} can also be written as
\begin{equation}
\label{shoterazoulai1}
[F(z),\eta]_{\mathcal C}=\langle(Pf)(b),\frac{\eta}{b-\overline{z}}\rangle_{\mathbf H_{2,r}^-(\mathcal C)}=[Pf, P(g_z\eta)]_P,
\end{equation}
see \eqref{esther1}. In particular we can write:
\[
[C_{\Phi^\sharp}(\cdot,w)\xi,C_{\Phi^\sharp}(\cdot,z)\eta]_{\mathcal L(\Phi^\sharp)}=[P(g_w\xi),P(g_z\eta)]_P.
\]

To obtain the required realization we first prove that the space ${\mathcal L(\Phi^\sharp)}$ is invariant under the backward shift operator defined by
\[
(R_0f)(z)=\begin{cases}\,\,\dfrac{f(z)-f(0)}{z}, z\not=0\\
\,\, f^\prime(0),\,\,\,\,\hspace{9.8mm} z=0,\end{cases}
\]
for vector-valued functions analytic in a neighborhood of the origin.

\begin{proposition}
Let $f\in {\mathbf H}_{2,r}^-(\mathcal C)$. We have
\begin{equation}
\label{manon}
[(R_0\widehat{Pf})(z),\eta]_{\mathcal C}=[Tf,\dfrac{\eta}{b-\overline{z}}]_{\Phi}=\langle (PTf)(b),\frac{\eta}{b-\overline{z}}\rangle_{\mathbf H_{2,r}^-(\mathcal C)},
\end{equation}
where $T$ is defined by \eqref{TTT}.
\label{prop42}
\end{proposition}
\begin{proof}
From \eqref{elizabeth} we have
\[
[(R_0\widehat{Pf})(z),\eta]_{\mathcal C}=[f,\frac{\eta}{b(b-\overline{z})}]_{\Phi}
=[(Tf)(a),\dfrac{\eta}{b-\overline{z}}]_{\Phi},
\]
and by definition of $P$, we have:
\[
[Tf,\dfrac{\eta}{b-\overline{z}}]_{\Phi}=\langle (PTf)(b),\frac{\eta}{b-\overline{z}}\rangle_{\mathbf H_{2,r}^-(\mathcal C)}.
\]
\end{proof}

The linear space ${\rm ran}\,P$ is dense in ${\rm ran}\,|P|^{1/2}$ in the $\|\cdot\|_P$ norm, but
\eqref{manon} does not show that $R_0$ extends to a bounded operator. This is proved now.

\begin{proposition}
\label{closed}
$R_0$ has a continuous extension to $\mathcal L(\Phi^\sharp)$.
\end{proposition}

\begin{proof}
The proof uses Loewner's theorem.
We have from \eqref{manon} and \eqref{rtyui} that
\[
\|R_0\widehat{Pf}\|^2=\langle |P|Tf,Tf\rangle_{\mathbf H_{2,r}^-(\mathcal C)}
\]
and
\[
\|\widehat{Pf}\|^2=\langle |P|f,f\rangle_{\mathbf H_{2,r}^-(\mathcal C)}
\]
To prove continuity of $R_0$ it is enough to show that there exists $k>0$ such that
\[
T^*|P|T\le k |P|.
\]
We  first show that there exists a constant $m>0$ such that
\begin{equation}
\label{123123123!!!}
(T^*|P|T)^2\le m P^2,
\end{equation}
and then use Loewner's theorem (see \cite{donoghue}) to get $(T^*|P|T)\le \sqrt{m} |P|$.
We have (using $PT=\mathsf{M}_{b^{-1}}P$; see \eqref{solange})
\[
\begin{split}
(T^*|P|T)^2&=T^*|P|\cdot(TT^*)\cdot|P|T\\
&\le T^*|P|(\|T\|^2I)|P|T\\
&= \|T\|^2 T^*|P|^2T\\
&= \|T\|^2 T^*P^2T\\
&=\|T\|^2 P\mathsf{M}_{b^{-1}}(\mathsf{M}_{b^{-1}})^*P\\
&\le \|\mathsf{M}_{b^{-1}}\|^2\|T\|^2 P^2,
\end{split}
\]
since the operator $\mathsf{M}_{b^{-1}}$ is bounded from $\mathbf H_{2,r}^-(\mathcal C)$
into itself. Thus $R_0$ has a continuous extension, say $X$. In fact $X$ is still defined as $R_0$. Indeed, if $(F_n)$ is a Cauchy sequence in $\mathcal L(\Phi^\sharp)$
converging to $F$ in an associated Hilbert space norm (they are all equivalent; see \cite{bognar}),
then $(R_0F_n)$ tends to $XF$ in the same norm. But in a reproducing kernel space, convergence in norm implies pointwise convergence, and so $XF(w)=R_0F(w).$

\end{proof}

\begin{proposition}
The adjoint of the operator $R_0$ is given by
\begin{equation}
R_0^{[*]}\widehat{(|P|^{1/2}f)}=\widehat{|P|^{1/2}T^*f},\quad f\in{\rm ran}\,(I-\pi),
\end{equation}
where $T$ is defined by \eqref{TTT} and $f\in {\mathbf H}_{2,r}^-(\mathcal C)$.
\end{proposition}

\begin{proof}
Since $R_0$ is bounded, so is its adjoint $R_0^{[*]}$. In particular the range of $R_0^{[*]}$ is inside $\mathcal L(\Phi^\sharp)$ and
 we can define $X$ via $R_0^{[*]}\widehat{|P|^{1/2}f}=\widehat{|P|^{1/2}Xf}$. We have:
\[
\begin{split}
[R_0^{[*]}\widehat{|P|^{1/2}f},\widehat{|P|^{1/2}g}]_{\mathcal L(\Phi^\sharp)}&=[R_0^{[*]}\widehat{|P|^{1/2}Xf},\widehat{|P|^{1/2}g}]_{\mathcal L(\Phi^\sharp)}\\
&=\langle Xf,\sigma(I-\pi)g\rangle_{\mathbf H_{2,r}^-(\mathcal C)}\\
&=\widehat{|P|^{1/2}Xf},R_0\widehat{|P|^{1/2}g}]_{\mathcal L(\Phi^\sharp)}\\
&=\langle f,\sigma(I-\pi)Tg\rangle_{\mathbf H_{2,r}^-(\mathcal C)}
\end{split}
\]
and so $\sigma(I-\pi)X=T^*\sigma(I-\pi)$.
\end{proof}

\begin{theorem}
\label{th55}
Let $\mathcal C$ be a Hilbert space, and let $\Phi$ be a $\mathbf L(\mathcal C,\mathcal C)$-valued function analytic in
$\mathbb D_{r_0}=$ with $r_0<1$) and continuous in $|z|\le r_0$ in the operator topology.
The space $\mathcal L(\Phi^\sharp)$ is $R_0$ invariant and $R_0$ is coisometric in this space.
Furthermore
Let $C$ denote the point evaluation at the origin. It is a bounded linear operator from the Krein space $\mathcal L(\Phi^\sharp)$ into the Hilbert space $\mathcal C$, and we denote
its adjoint by $C^{[*]}$. We have
\begin{equation}
\label{rehovyavne}
\Phi(z)=i{\rm Im}\,\Phi(0)+\frac{1}{2}C(I-zR_0^{[*]})(I+zR_0^{[*]})^{-1}C^{[*]}
\end{equation}
Finally any such realization
\[
\Phi(z)=i{\rm Im}\,\Phi(0)+\frac{1}{2}C(I-zV^{[*]})(I+zV^{[*]})^{-1}D^{[*]}
\]
 with a coisometric operator $V$ acting in a Krein space $\mathcal K$, and closely outer connected, meaning that the span of the operators
$(I-zV)^{-[*]}D^{[*]}$ is dense in the space,
is unique up to a weak isomorphism.
\end{theorem}

\begin{proof} From \eqref{muriel} we have
(on a dense set)
\[
\langle TPf,Pg\rangle_P=\langle Pf, P\mathsf{M}_{b^{-1}}g\rangle_P.
\]

Thus $T^{[*]}Pg=P\mathsf{M}_{b^{-1}}g$ and so, using \eqref{solange}
\[
TT^{[*]}Pg=TP\mathsf{M}_{b^{-1}}g=TT^*g.
\]
Since in ${\mathbf H_{2,r}^-(\mathcal C)}$ we have $T^*=\mathsf{M}_{b^{-1}}$ we have $TT^*=I$ on the range of $I-\pi$ and so $TT^{[*]}P=I$ there.
It follows that $R_0R_0^{[*]}=I$ in $\mathcal L(\Phi^\sharp)$.\\

Let $C$ denote the evaluation at the origin. We have for $f\in\mathcal L(\Phi^\sharp)$ with power series expansion $f(u)=\sum_{n=0}^\infty f_nz^n$
\[
f_n=CR_0^nf
\]
and so
\[
f(z)=C(I-zR_0)^{-1}f
\]
Applying this equality to $f=C^{[*]}\xi$ we obtain
\[
((\Phi^\sharp(z)+(\Phi^\sharp)^{*}(0))=C(I-zR_0)^{-1}C^{[*]}.
\]
In particular $2{\rm Re}\,\Phi^\sharp(0)=CC^{[*]}$, and we have
\[
\begin{split}
\Phi^\sharp(z)+(\Phi^\sharp(0))^*&=C(I-zR_0)^{-1}C^{[*]}\\
\frac{1}{2}(\Phi^\sharp(0)+(\Phi^\sharp(0))^*)&=\frac{1}{2}CC^{[*]}
\end{split}
\]
and so
\[
\begin{split}
\Phi^\sharp(z)+\frac{1}{2}((\Phi^\sharp(0))^*-\Phi^\sharp(0))&=C(I-zR_0)^{-1}C^{[*]}-\frac{1}{2}CC^{[*]}\\
&=\frac{1}{2}C(I-zR_0)^{-1}(I+zR_0)C^{[*]}
\end{split}
\]
and \eqref{rehovyavne} follows.\\

The uniquenesss follows from the representation
\[
K(z,w)=C(I-zR_0)^{-1}(I-wR_0)^{-[*]}C^{[*]}
\]
for the reproducing kernel of $\mathcal L(\Phi^\sharp)$. A more detailed argument
is presented in the quaternionic setting at the end of the paper.
\end{proof}

\begin{remark}
\label{lastrem}
{\rm This remark is set in the notation of Theorem \ref{th55}.
One obtains the result for the case where $\mathcal C$ is a Krein space, endowed with a form $[\cdot,\cdot]_{\mathcal C}$ as follows. Let $J$ be a fundamental symmetry such that
$(\mathcal C,[J\cdot,\cdot]_{\mathcal C})$ is a Hilbert space. Note that
\begin{align}
[\xi,\eta]_{\mathcal C}&=\langle J\xi,\eta\rangle_{\mathcal C}\\
\nonumber
&\mbox{\hspace{-5.1cm}{\rm and}}
\\
[J\xi,\eta]_{\mathcal C}&=\langle \xi,\eta\rangle_{\mathcal C},\quad \xi,\eta\in\mathcal C.
\end{align}
Recall also that the Krein space adjoint and the Hilbert space adjoint of an operator
$X\in\mathbf L(\mathcal C,\mathcal C)$ are linked by
\begin{equation}
X^{[*]}=JX^*J.
\end{equation}
Applying Theorem \ref{th55} to the function $\Phi(z)J$ will lead to the Krein space result since
\[
\begin{split}
[\frac{\Phi^\sharp(z)+(\Phi^\sharp(w))^{[*]}}{1-z\overline{w}}\xi,\eta]_{\mathcal C}&=\langle\frac{\Phi^\sharp(z)+J(\Phi^\sharp(w))^{*}J}{1-z\overline{w}}\xi,J\eta\rangle_{\mathcal C}\\
&=\langle\frac{J\Phi^\sharp(z)+(\Phi^\sharp(w))^{*}J}{1-z\overline{w}}\xi,J\eta\rangle_{\mathcal C}\\
&=\langle\frac{(\Phi J)^\sharp(z)+((\Phi J)^\sharp(w))^{*}}{1-z\overline{w}}\xi,J\eta\rangle_{\mathcal C}.
\end{split}
\]
Note the operator $C$ can now be viewed as an operator between two Krein spaces, or as an operator from a Krein space into the Hilbert space $\mathcal C$. These adjoints are
related by multiplication by the operator $J$ on the right.}
\end{remark}

\section{The quaternionic setting: Slice hyperholomorphic functions}
\setcounter{equation}{0}
\label{quater}
In this section we collect some basic notations and notions useful in the sequel.
For more details, we refer the interested reader to the books \cite{zbMATH06658818,MR2752913,MR3013643}.
\\

By $\mathbb{H}$ we denote the algebra of real quaternions.
The imaginary units $i$, $j$ and $k$ in $\mathbb{H}$ satisfy $ijk=-1$, $i^2=j^2=k^2=-1$.
 An element in $\mathbb{H}$ is of the form $p=x_0+ix_1+jx_2+kx_3$, where $x_\ell\in \mathbb{R}$.
The real part, the imaginary part and the modulus of a quaternion are defined as
${\rm Re}(q)=x_0$, ${\rm Im}(q)=i x_1 +j x_2 +k x_3$, $|q|^2=x_0^2+x_1^2+x_2^2+x_3^2$, respectively.
The conjugate of the quaternion $q=x_0+ix_1+jx_2+kx_3$ is
$\bar q={\rm Re }(q)-{\rm Im }(q)=x_0-i x_1-j x_2-k x_3$
 and it satisfies
$$
|q|^2=q\bar q=\bar q q.
$$
By $\mathbb S$ we denote the unit sphere of purely imaginary quaternions
$$
\mathbb{S}=\{q=ix_1+jx_2+kx_3\ {\rm such \ that}\
x_1^2+x_2^2+x_3^2=1\}.
$$
Note that if $I\in\mathbb{S}$, then
$I^2=-1$; for this reason the elements of $\mathbb{S}$ are also called
imaginary units.
Given a nonreal quaternion $q=x_0+{\rm Im} (q)=x_0+I |{\rm Im} (q)|$, $I={\rm Im} (q)/|{\rm Im} (q)|\in\mathbb{S}$, we can associate to it the 2-dimensional sphere defined by
$$
[q]=\{x_0+I  {\rm Im} (q)|\  :\  \ I\in\mathbb{S}\}.
$$
This sphere has center at the real point $x_0$ and radius $|{\rm Im} (q)|$.
An element in the complex plane $\mathbb{C}_I=\mathbb{R}+I\mathbb{R}$ is denoted by $x+Iy$.
\\
A subset $\Omega$ of $\mathbb H$ is said to be axially symmetric if $x+Jy\in\Omega$ for all $J\in\mathbb S$ whenever $x+Iy\in\Omega$ for some $I\in\mathbb S$.
\\

Let $p$ be a quaternion. It can be written as
\[
q=z_1+z_2j
\]
where $z_1=x_0+ix_1$ and $z_2=x_2+ix_3$ are complex numbers. The map
\begin{equation}
\label{chixhixhi}
\chi(q)=\begin{pmatrix}z_1&z_2\\-\overline{z_2}&\overline{z_1}\end{pmatrix}
\end{equation}
is a skew-field homomorphism, extended to quaternionic matrices $M=A+jB$ (with $M\in\mathbb H^{n\times n}$ and $A,B\in\mathbb C^{n\times n}$) by
\begin{equation}
\label{chi890}
\chi(M)=\begin{pmatrix}A&B\\
-\overline{B}&\overline{A}\end{pmatrix}.
\end{equation}

\begin{remark}{\rm 
The map $\chi$ can furthermore be extended to bounded
operators, a fact we use in the proof of Theorem \ref{th56}. We note that
\begin{equation}
\label{chi891}
\chi(M)\ge 0\quad\iff\quad M\ge 0
\end{equation}
(both for a matrix or an operator)
and that
\[
M_1\le M_2\quad\iff \chi(M_1)\le \chi(M_2)
\]
for hermitian matrices (or operators) $M_1$ and $M_2$.}
\label{karl}
\end{remark}

We now introduce the notion of slice hyperholomorphic functions with values in a two-sided quaternionic Banach space $\mathcal B$.  In particular, the definition includes the case of functions with values in $\mathbb H$, see \cite{MR2737796}.
\begin{definition}
Let $\mathcal B$ be a two-sided quaternionic Banach space and $\Omega$ be an axially symmetric open set.
Let $f$ be a function of the form
$f(p)=f(x+{I}y)=\alpha (x,y) +{I}\beta (x,y)$ where $\alpha, \beta:
\Omega\to \mathcal{B}$ depend only on $x,y$, are real differentiable, satisfy the Cauchy-Riemann
equations
\begin{equation}\label{CR1}
\begin{cases}
\partial_x \alpha -\partial_y\beta=0\\
\partial_y \alpha
+\partial_x\beta=0,
\end{cases}
\end{equation}
and let us assume
\begin{equation}\label{alfabeta1}
\begin{split}
\alpha(x,-y)=\alpha(x,y),\qquad
\beta(x,-y)=-\beta(x,y).
\end{split}
\end{equation}
Then $f$ is said to be left slice hyperholomorphic. If, under the same hypothesis, $f$ is of the form
$f(p)=f(x+{I}y)=\alpha (x,y) +\beta (x,y)I$, then it is said to be right slice hyperholomorphic.
\end{definition}

The class of slice hyperholomorphic quaternionic valued functions is important since power series centered at real points are slice hyperholomorphic. Let us denote by
$\mathbb B_R$ the open ball centered at $0$ and radius $R>0$. A function
$f:\,\mathbb B_R \to \mathcal{B}$ is  left slice regular if and only if $f$ admits power series expansion
$$
f(q)=\sum_{m=0}^{+\infty} q^m f_m,\quad f_m\in\mathcal B
$$
converging on $\mathbb B_R$.
\\
The pointwise multiplication of two slice regular functions is not slice regular, in general. Instead, we introduce the following product:
\begin{definition}
Let $\Omega\subseteq\mathbb{H}$ be an axially symmetric and let
$f,g:  \Omega\to \mathcal{B}$ be slice hyperholomorphic functions with values in a two sided  quaternionic Banach algebra $\mathcal B$. Let
$f(x+{I}y)=\alpha(x,y)+{I}\beta(x,y)$,
$g(x+{I}y)=\gamma(x,y)+{I}\delta(x,y)$. Then we define
\begin{equation}\label{opleft}
(f\star g)(x+{I}y):= (\alpha\gamma -\beta \delta)(x,y)+
{I}(\alpha\delta +\beta \gamma )(x,y).
\end{equation}
Similarly we can define a $\star_r$ product between right slice hyperholomorphic functions.
\end{definition}
It can be verified that the function $f\star g$ is slice
hyperholomorphic.
\begin{remark}{\rm We note that  if $f(p)=\sum_{n= 0}^\infty
p^n f_n$ and $g(p)=\sum_{n= 0}^\infty p^n f_n$, with $f_n,g_n\in \mathcal{B}$ for all
$n$, then
\[
(f\star g)(p):=\sum_{n= 0}^\infty p^n (\sum_{r=0}^n
f_rg_{n-r}).
\]
}
\end{remark}
For scalar valued functions, it is possible to define an inverse with respect to the (left or right) $\star$-product. In this paper we are only interested in defining the $\star_r$-inverse of a function of the form $I-p T$ and we will limit ourselves to this case.
\\
If we consider the function $(1-pq)^{-\star_r}$ and we use the functional calculus, we can define
$(1-pT)^{-\star_r}$. Note that for $p\not=0$
$$
(1-pT)^{-\star_r}=p^{-1}S_R(p,T)=-p^{-1}(T-\overline{s}{I})(T^2-2 {\rm Re}\, ( s) T+|s|^2{I})^{-1},
$$
moreover
$$
(1-pT)^{-\star_r}=\sum_{n\geq 0} p^nT^n \qquad {\rm for} \ \ |p|\|T\|<1 .
$$
For the sake of simplicity,  in the sequel we will write $(1-sT)^{-\star}$.

\section{The quaternionic setting:  The realization theorem}
\setcounter{equation}{0}
\label{quater2}
Our goal is to prove the counterpart of Theorem \ref{th55} in the quaternionic setting, in the framework of slice hyperholomorphic functions. The coefficient space $\mathcal C$
is now a two-sided quaternionic Hilbert space (the case when $\mathcal C$ is a two-sided Krein space is considered at the end of the section). The inner product in
$\mathcal C$ is moreover assumed to satisfy the following condition:
\begin{equation}
\label{conditioncoluche}
\langle c, qd\rangle_{\mathcal C}=\langle \overline{q} c ,d\rangle_{\mathcal C},\quad \forall c,d\in\mathcal C\quad{\rm and}\quad \forall q\in\mathbb H.
\end{equation}
We note that, in general, one has Hilbert spaces on one side, say  on the right. By fixing a Hilbert basis it is possible to define a multiplication by a scalar also on the left, thus showing the Hilbert space as a two-sided quaternionic vector space.
\\
Our starting point is a
$\mathbf L(\mathcal C,\mathcal C)$-valued function $\Phi$, left slice hyperholomorphic in $p$ in ${\mathbb B_r}$, and bounded in $\overline{\mathbb B_{r_0}}$ where $r_0>r$. The kernel $C_\Phi$ of the
complex setting now becomes
\begin{equation}
\label{kphi}
K_\Phi(p,q)=\sum_{u,v=0}^\infty p^u(\Phi(p)+\Phi(q)^*)\overline{q}^v.
\end{equation}
Note that this kernel can be written in closed form as
$$
K_\Phi(p,q)=(\Phi(p)+\Phi(q)^*)\star(1-p\overline{q})^{-\star}
$$
where the $\star$-multiplication is computed in the variable $p$.
We set
\[
\Phi^\sharp(p)=\Phi_0^*+p\Phi_1^*+p^2\Phi_2^*+\cdots
\]
 and we build a reproducing kernel Krein space $\mathcal L(\Phi^\sharp)$ with reproducing kernel $K_{\Phi^\sharp}$; see Definition \ref{65}.
With this space $\mathcal L(\Phi^\sharp)$ at hand, the main result of the section, whose proof is postponed to section 7, is:

\begin{theorem}
\label{th56}
Let $\mathcal C$ be a two-sided quaternionic Hilbert space
and let $\Phi$ be a $\mathbf L(\mathcal C,\mathcal C)$-valued function slice hyperholomorphic in
$\mathbb B_{r_0}$ with $r_0<1$ and continuous in $|p|\le r_0$ in the operator topology. We set
\[
\Phi(p)=\Phi_0+p\Phi_1+p^2\Phi_2+\cdots
\]
The space $\mathcal L(\Phi^\sharp)$ is $R_0$ invariant and $R_0$ is coisometric in this space.
Furthermore
Let $C$ denote the point evaluation at the origin. It is a bounded linear operator from the Krein space $\mathcal L(\Phi^\sharp)$ into the Hilbert space $\mathcal C$; denote
its adjoint by $C^{[*]}$. We have
\begin{equation}
\label{rehovyavne1}
\Phi^\sharp (p)=-i{\rm Im}\,\Phi(0)+\frac{1}{2}C\star (I-pR_0)(I+pR_0)^{-\star}C^{[*]} .
\end{equation}
Finally any such realization
\begin{equation}
\label{real2}
\Phi^\sharp (p)=-i{\rm Im}\,\Phi(0)+\frac{1}{2}G\star (I-pV)(I+V)^{-\star}G^{[*]}
\end{equation}
with a coisometric operator defined in a quaternionic Krein space $\mathcal K$,
and closely outer connected, meaning that the span of the operators $(I-pV^{[*]})^{-\star}D^{[*]}$ is dense in the space, is unique up to a weak isomorphism.
\end{theorem}

 The spectral theorem for bounded Hermitian
operator is still true in the quaternionic setting (see \cite[\S 8, p. 57]{acs_OT_244}), and it follows that the results of Section \ref{lydia} still hold for quaternionic
bounded Hermitian operators. More precisely, if
\[
P=\int_{\mathbb R} \lambda dE(\lambda)
\]
is the integral representation of $P$ along its spectral measure $E$, we set
\[
\begin{split}
|P|&=\int_{\mathbb R} |\lambda| dE(\lambda)\\
|P|^{1/2}&=\int_{\mathbb R} \sqrt{|\lambda|} dE(\lambda)\\
\sigma&=\int_{\mathbb R\setminus\left\{0\right\}} \frac{\lambda}{|\lambda|} dE(\lambda).
\end{split}
\]

We define $\ell_{2,r}(\mathbb N,\mathcal C)$ to be the space of vectors
\[
{\mathbf f}=\begin{pmatrix}f_1\\ f_2\\ \vdots\end{pmatrix}\in{\mathcal C}^{\mathbb N}
\]
such that \eqref{concorde} holds, i.e.:
\begin{equation*}
\sum_{u=1}^\infty R^{2u}\|f_u\|_{\mathcal C}^2<\infty.
\end{equation*}

To such a sequence one associates the function
\begin{equation}
\label{place_d_italie}
f(p)=\sum_{u=1}^\infty p^{-u}f_u
\end{equation}
which is left slice hyperholomorphic in $|p|>r$. As in the complex case we denote this space by $\mathbf H^-_{2,r}(\mathcal C)$.\\

We define an Hermitian form as in Proposition \ref{illusion}:
\begin{equation}
\label{hermi}
[f,g]_\Phi=\sum_{v=1}^\infty\sum_{u=1}^v \langle \Phi_{v-u}f_v,g_u\rangle_{\mathcal C}+
\sum_{u=1}^\infty\sum_{v=1}^u \langle \Phi^*_{u-v}f_v,g_u\rangle_{\mathcal C}.
\end{equation}

\begin{proposition}
The form \eqref{hermi} is jointly continuous on $\ell_{2,r}(\mathbb N,\mathcal C)\times \ell_{2,r}(\mathbb N,\mathcal C)$.
\end{proposition}

\begin{proof}
We first consider the form
\begin{equation}
\label{hermi1}
[f,g]_1=\sum_{v=1}^\infty\sum_{u=1}^v \langle \Phi_{v-u}f_v,g_u\rangle_{\mathcal C}.
\end{equation}
The series $\sum_{u=0}^\infty p^u\Phi_u$ converges in the operator norm in $|z|\le r_0$. Hence, given $0<r<r_0$ there exists $M_r$
be such that
\begin{equation}
\|\Phi_v\|r^v\le M_r.
\end{equation}

We have
\begin{equation}
\label{wow}
\begin{split}
\big|\sum_{v=1}^\infty\sum_{u=1}^v \langle \Phi_{v-u}f_v,g_u\rangle_{\mathcal C}\big|&\le\sum_{v=1}^\infty\sum_{u=1}^v\|\Phi_{v-u}\|\cdot\|f_v\|_{\mathcal C}\cdot
\|g_u\|_{\mathcal C}\\
&=\sum_{u=1}^\infty\|g_u\|_{\mathcal C}\left(\sum_{v=u}^\infty \|\Phi_{v-u}\|\cdot\|f_v\|_{\mathcal C}\right)\\
&\le \sum_{u=1}^\infty\|g_u\|_{\mathcal C}\left(\sum_{v=u}^\infty M\cdot R^{v-u}\|f_v\|_{\mathcal C}\right)\\
&=M_r\sum_{u=1}^\infty\|g_u\|_{\mathcal C}(R^ur^u)r^u\left(\sum_{v=u}^\infty M\cdot (r^vR^v)R^v\|f_v\|_{\mathcal C}\right)\\
&\le M_r\|g\|_{\mathbf H_{2,r}^-(\mathcal C)}\cdot\sqrt{\frac{1}{1-r^2}}\cdot K_{f,r}
\end{split}
\end{equation}
with
\[
K_{f,r}=\sqrt{\frac{1}{1-r}}\left(\sum_{v=1}^\infty R^{4v}\|f_v\|^2\right)^{1/2}.
\]
Thus the form \eqref{hermi1} is continuous in $g$. The above inequalities will not show that the form is jointly continuous since $K_{f,r}$ is not the
${\mathbf H}_{2,r}^-(\mathcal C)$-norm of $f$ (notice a term $R^4$ rather than $R^2$ appearing in $K_{f,r}$). Rewriting the first line in \eqref{wow} as
\[
\begin{split}
\big|\sum_{v=1}^\infty\sum_{u=1}^v \langle \Phi_{v-u}&f_v,g_u\rangle_{\mathcal C}\big|\le\sum_{v=1}^\infty\sum_{u=1}^v\|\Phi_{v-u}\|\cdot\|f_v\|_{\mathcal C}\cdot\|g_u\|_{\mathcal C}\\
&=\sum_{v=1}^\infty\|f_v\|_{\mathcal C}\cdot\left(\sum_{u=1}^v\|\Phi_{v-u}\|\cdot\|g_u\|_{\mathcal C}\right)\\
&\le \sum_{v=1}^\infty\|f_v\|_{\mathcal C}\\
&\le M_r\sum_{v=1}^\infty\|f_v\|_{\mathcal C}\cdot r^v(r^vR^v) \left(\sum_{u=1}^v\|R^u(R^ur^u)\|\cdot\|g_u\|_{\mathcal C}\right)
\end{split}
\]
and hence, by the same argument, the form \eqref{hermi1} is continuous in $f$. It follows from the uniform boundedness theorem (which holds in quaternionic Hilbert spaces;
see \cite[Theorem 3.3 p. 39]{acs_OT_244}) that the form \eqref{hermi} is bounded.
\end{proof}

\begin{remark}{\rm
For a general result in the complex setting on separately continuous forms which are jointly continuous, see \cite[IV.26, Theorem 2]{MR83k:46003} and the discussion
in \cite[p. 216]{MR3404695}.}
\end{remark}

As in the complex setting case, Riesz representation theorem implies the existence of a bounded operator, which we still call $P$, such that
\begin{equation}
\label{610}
[f,g]_\Phi=\langle Pf,g\rangle_{\ell_{2,r}(\mathbb N,\mathcal C)}.
\end{equation}
As in the previous section, we have $P={\rm Re}\, T_\Phi$, where the block-Toeplitz operator $T_\Phi$ is defined by \eqref{ttt}.

\begin{proposition}
\label{sabine}
Let $\mathsf{M}_{p^{-1}}$ and $T$ be defined in the Hardy space of slice hyperholomorphic functions $\mathbf H^-_{2,r}(\mathcal C)$ by
\begin{align}
(\mathsf{M}_{p^{-1}}f)(p)&=p^{-1}f(p)\,\,\hspace{6.2mm}=\,\,\sum_{n=2}^\infty p^{-u}f_{u-1}\\
(Tf)(p)&=pf(p)-f_1\,\hspace{2.8mm}=\,\,\,\sum_{u=1}^\infty p^{-u}f_{u+1}.
\end{align}
Then,
\begin{equation}
\label{613}
[Tf,g]_\Phi=[f,\mathsf{M}_{p^{-1}}g]_\Phi.
\end{equation}
\end{proposition}

The proof is a direct computation on the coefficients, and will be omitted here.\\

The reproducing kernel Krein space of left slice hyperholomorphic functions with reproducing kernel \eqref{kphi}, is constructed as in the previous section, but with the operator
${\rm Re}\, T_\Phi$ and using the unitary map \eqref{place_d_italie} between $\ell_{2,r}(\mathbb N,\mathcal C)$ and ${\mathbf H}_{2,r}^-(\mathcal C)$. We take
$\mathcal H=\ell_{2,r}(\mathbb N,\mathcal C)$ in \eqref{michelle} and \eqref{michelle1}, and for $a\in\mathbb B_r$ we set in \eqref{widehat} $g_a$ to be the operator from $\mathcal C$ into
$\mathcal H$ defined by
\begin{equation}
\label{two}
g_ac=\begin{pmatrix}c\\ \overline{a}c\\ \overline{a}^2c\\ \vdots\end{pmatrix}.
\end{equation}
Note that the hypothesis that the coefficient space is two-sided is used in \eqref{two}. The operator $g_a$ is well defined since $\mathcal C$ is left quaternionic, and right linear
because $\mathcal C$ is right quaternionic.

\begin{definition}
The space $\mathcal L(\Phi^\sharp)$ is the set of functions $F$ from $\mathbb B_r$ into $\mathcal C$ defined by
\begin{equation}
\langle F(a),c\rangle_{\mathcal C}=\langle |P|^{1/2}f,g_ac\rangle_{\ell_{2,r}(\mathbb N,\mathcal C)},
\end{equation}
endowed with the associated forms \eqref{shulamit111}-\eqref{eglantine}. As in \eqref{widehat} we set
\begin{equation}
F(a)=\widehat{|P|^{1/2}f}(a).
\end{equation}
\label{65}
\end{definition}

\section{Proof of Theorem \ref{th56}}
\setcounter{equation}{0}
\label{sec7}
 Theorem \ref{th56}, the realization theorem, is proved along the lines of the proof of Theorem \ref{th55} using the space $\mathcal L(\Phi^\sharp)$ and we proceed in a
number of steps.\\

STEP 1: {\sl The space $\mathcal L(\Phi^\sharp)$ is backward-shift invariant, where the shift is now defined
by
\[
(R_0f)(p)=p^{-1}(f(p)-f(0))=f_1+pf_2+\cdots
\]
for $f(p)=f_0+pf_1+p^2f_2\cdots$.}\\

This follows from Proposition \ref{sabine}. In fact, using \eqref{610} and \eqref{613} we have the counterpart of \eqref{123123123!!!} for quaternionic operators. 
By Remark \ref{karl} we can use the map $\chi$ defined in \eqref{chixhixhi} and \eqref{chi890}, and we have that
\[
(\chi(T^*|P|T))^2\le m (\chi(P))^2.
\]
Applying Loewner's theorem to these operators, and by uniqueness of the positive squareroot of a positive operator, we obtain
\[
\chi(T^*|P|T)\le \sqrt{m} \chi(P)
\]
and so
\[
(T^*|P|T)\le \sqrt{m} P
\]
in view of \eqref{chi890}-\eqref{chi891}. The argument is then as in the proof of Proposition \ref{closed}.\\

STEP 2: {\sl We have
\begin{equation}
(\widehat{Pg_ac})(p)=K_\Phi(p,a)c
\end{equation}
and the space $\mathcal L(\Phi^\sharp)$ has reproducing kernel $K_{\Phi^\sharp}$.}\\

Indeed, for $a,p\in\mathbb B_r$ and $c,d\in\mathcal C$ and using the condition \eqref{conditioncoluche} satisfied by the inner product in $\mathcal C$, we can write:
\[
\begin{split}
\langle (\widehat {Pg_ac})(p),d\rangle_{\mathcal C}&=[g_ac,g_pd]_\Phi\\
&=\sum_{v=1}^\infty\sum_{u=1}^\infty\langle \Phi_{v-u}\overline{a}^{v-1}c,\overline{p}^{u-1}d\rangle_{\mathcal C}+\\
&\hspace{5mm}+\sum_{u=1}^\infty\sum_{v=1}^u\langle \Phi_{u-v}^*\overline{a}^{v-1}c,\overline{p}^{u-1}d\rangle_{\mathcal C}\\
&=\sum_{u=1}^\infty \langle p^{u-1}\left(\sum_{v=u}^\infty\Phi_{v-u}\overline{a}^{v-u}\right)\overline{a}^{u-1}c,d\rangle_{\mathcal C}+\\
&\hspace{5mm}+\sum_{v=1}^\infty\langle p^{v-1}\left(\sum_{v=u}^\infty p^{u-v} \Phi_{u-v}^*\right)\overline{a}^{v-1}c,d\rangle_{\mathcal C}\\
&=\sum_{u=1}^\infty\langle p^{u-1}(\Phi^\sharp(a))^* \overline{a}^{u-1}c,d\rangle_{\mathcal C}+\\
&\hspace{5mm}+\sum_{u=1}^\infty\langle p^{u-1}\left(\Phi^\sharp(p)\right)\overline{a}^{u-1}c,d\rangle_{\mathcal C}\\
&=\langle K_\Phi(p,a)c,d\rangle_{\mathcal C}.
\end{split}
\]

We then apply Proposition \ref{z1z2z3}.\\

STEP 3: {\sl Let $C$ denote the point evaluation at the origin. We have}
\begin{align}
(C^{[*]}c)(p)&=(\Phi^\sharp(p)+\Phi(0))p,\quad c\in\mathcal C,\\
F(p)&=C\star(I-pR_0)^{-\star}F, \quad F\in\mathcal L(\Phi^\sharp).
\end{align}

For $F\in\mathcal L(\Phi^\sharp)$, with power series expansion
\[
F(p)=\sum_{u=0}^\infty p^uF_u,\quad F_0,F_1,\ldots\in\mathcal C
\]
we have,
\[
F_n=CR_0^nF,\quad n=0,1,\ldots,
\]
and so
\begin{equation}
F(p)=\sum_{u=0}^\infty p^uCR_0^nF=C\star(I-pR_0)^{-\star}F.
\end{equation}
For $c\in\mathcal C$ the function $p\mapsto (\Phi^\sharp(p)+\Phi(0))c$ belongs to $\mathcal L(\Phi^\sharp)$. In a way similar to the complex numbers setting we can write
\[
(\Phi^\sharp(p)+\Phi(0))c=C\star(I-pR_0)^{-\star}C^{[*]}c
\]
and in particular
\[
(\Phi^*(0)+\Phi(0))c=CC^{[*]}c
\]
Hence
\[
\begin{split}
(\Phi^\sharp c)(p)&=C\star(I-pR_0)^{-\star}C^{[*]}c+\frac{1}{2}\left(\Phi(0)^*-\Phi(0)-(\Phi(0)^*+\Phi(0))\right)c\\
&=C\star(I-pR_0)^{-\star}C^{[*]}c-\frac{1}{2}CC^{[*]}c-i({\rm Im}\, \Phi(0))c\\
&=\frac{1}{2}C\star(I-pR_0)^{-\star}\star(I+pR_0)\star C^{[*]}c-i({\rm Im}\, \Phi(0))c,
\end{split}
\]
and hence we get \eqref{rehovyavne1}.\\

STEP 4: {\sl We prove the uniqueness of the realization up to a weak isomorphism.}\\

We define two densely defined relations $R_1$ and $R_2$ in
$\mathcal L(\Phi^\sharp)\times\mathcal K$ and $\mathcal K\times\mathcal L(\Phi^\sharp)$ respectively by the span of pairs of the form
\[
(((I-pR_0)^{-\star}C^{[*]}c),(I-pV)^{-\star}C^{[*]}c)
\]
and
\[
(((I-pV)^{-\star}D^{[*]}c),(I-pR_0)^{-\star}C^{[*]}c),\quad c\in\mathcal C.
\]
Note that $R_1=R_2^{-1}$. From
\[
F_n^*=CR_0^nC^{[*]}=DR_0^nD^{[*]},
\]
we see that $R_1$ and $R_2$ are densely defined, and thus are graphs of densely defined isometries $W_1$ and $W_2$ from
$\mathcal L(\Phi^\sharp)$ into $\mathcal K$ and from $\mathcal K$ by into $\mathcal L(\Phi^\sharp)$ with dense ranges and
respectively defined by
\[
\begin{split}
W_1((I-pR_0)^{-\star}C^{[*]}c)&=(I-pV)^{-\star}C^{[*]}c\\
W_2((I-pV)^{-\star}D^{[*]}c)&=(I-pR_0)^{-\star}C^{[*]}c,\quad c\in\mathcal C.
\end{split}
\]
In general one cannot extend $W_1$ or $W_2$ to unitary continuous mappings since we are in the
Krein space setting. Such an extension will be possible in the case of Pontryagin, and in particular Hilbert, spaces.
Note that $W_1W_2=I$ and $W_2W_1=I$ on dense subspaces of $\mathcal K$ and $\mathcal L(\Phi^\sharp)$ respectively.

\begin{remark}
{\rm The case where the coefficient space $\mathcal C$ is a two-sided quaternionic Krein space (as opposed to a two-sided quaternionic Hilbert space) is treated as in the
complex setting; see Remark \ref{lastrem}.}
\end{remark}

\bibliographystyle{plain}
\def\cprime{$'$} \def\cprime{$'$} \def\cprime{$'$}
  \def\lfhook#1{\setbox0=\hbox{#1}{\ooalign{\hidewidth
  \lower1.5ex\hbox{'}\hidewidth\crcr\unhbox0}}} \def\cprime{$'$}
  \def\cprime{$'$} \def\cprime{$'$} \def\cprime{$'$} \def\cprime{$'$}
  \def\cprime{$'$}

\end{document}